\title{\vspace{-0.7cm}Universality of random graphs and rainbow embedding}
\author{ Asaf Ferber
\and Rajko Nenadov \and Ueli Peter \thanks{Institute of Theoretical
Computer Science ETH, 8092 Z\"urich, Switzerland. Emails:
asaf.ferber@inf.ethz.ch, rnenadov@inf.ethz.ch and
upeter@inf.ethz.ch}}

\documentclass[12pt]{article}
\usepackage{amsmath,amssymb,latexsym,color,epsfig,enumerate,a4}

\newif\ifnotesw\noteswtrue% T to show comments; F supresses.

\newcommand{\gnpk}{\ensuremath{\mathcal G_c(n,p)}}
\newcommand{\gnp}{\ensuremath{\mathcal G(n,p)}}

\newtheorem{theorem}{Theorem}[section]
\newtheorem{lemma}[theorem]{Lemma}
\newtheorem{claim}[theorem]{Claim}

\newtheorem{observation}[theorem]{Observation}
\newtheorem{corollary}[theorem]{Corollary}

\newtheorem{definition}[theorem]{Definition}

\newenvironment{proof}{\noindent{\bf Proof\,}}{\hfill$\Box$}

\begin{document}
\maketitle

\begin{abstract}
In this paper we show how to use simple partitioning lemmas in order
to embed spanning graphs in a typical member of $\gnp$. Let the \emph{maximum density} of a graph $H$ be the maximum average degree of all the subgraphs of $H$.
First, we
show that for $p=\omega(\Delta^{12} n^{-1/2d}\log^3n)$, a graph
$G\sim \gnp$ w.h.p.\ contains copies of all spanning graphs $H$ with
maximum degree at most $\Delta$ and maximum density at most $d$.
For $d<\Delta/2$, this improves a result of Dellamonica,
Kohayakawa, R\"odl and Ruci\'ncki. Next, we show that if we additionally restrict the spanning graphs to have girth at least $7$ then the random graph contains w.h.p.\ all such graphs for $p=\omega(\Delta^{12} n^{-1/d}\log^3n)$.
In particular, if $p=\omega(\Delta^{12} n^{-1/2}\log^3 n)$, the random graph therefore contains w.h.p.\ every spanning tree with
maximum degree bounded by $\Delta$. This improves a result of
Johannsen, Krivelevich and Samotij.

Finally, in the same spirit, we
show that for any spanning graph $H$ with constant maximum degree, and for suitable $p$, if we randomly color the edges
of a graph $G\sim \gnp$ with $(1 + o(1))|E(H)|$ colors, then w.h.p.\
there exists a \emph{rainbow} copy of $H$ in $G$ (that is, a copy of
$H$ with all edges colored with distinct colors).
\end{abstract}

\section{Introduction}

A graph $G$ is \emph{universal} for a family of graphs $\mathcal H$
(we write $G$ is $\mathcal H$-universal), if $G$
contains a copy of every graph $H\in \mathcal H$. The construction
(explicit and/or randomized) of sparse universal graphs for various
families has received a considerable amount of attention (see
\cite{ABHKP,AA,AC,ACKRRS,AKS,BLC,BCS,BCLR,BPT,C,CG,DKRR,JKS,KL}).

In particular, the probability space $\gnp$ of all graphs
on $n$ vertices, in which each pair of vertices forms an edge with
probability $p$ independently at random, has been considered in many papers. The problem of finding for
which values of $p$ a typical member of $\gnp$ is $\mathcal
H$-universal for various families of graphs is fundamental in the
theory of random graphs.

Let $\mathcal H(n,\Delta,d)$ be
the family of all graphs on $n$ vertices with maximum degree at most
$\Delta$ and with \emph{maximum density} at most $d$, where the maximum density of a
graph $G$ (denoted by $d(G)$) is defined as
$$d(G)=\max\left\{\frac{2|E(H)|}{|V(H)|}: H\subseteq G\right\}.$$

Dellamonica, Kohayakawa, R\"odl and Ruci\'ncki proved in
\cite{DKRR} that for maximum degree $\Delta\geq 3$ and an
edge probability $p=\omega\left(n^{-1/\Delta} \log^{1/\Delta}
n\right)$, a typical member of $\gnp$ is $\mathcal
H(n,\Delta,\Delta)$-universal. Recently, Kim and Lee \cite{KL}
obtained similar bounds for $\Delta=2$. In the following theorem we show that
if $d<\Delta/2$, then the bound in \cite{DKRR} can be further
improved.

\begin{theorem}\label{main}
Let $n$ be a positive integer, and let $\Delta =
\Delta(n)
> 1$ and $d=d(n)\geq2$ be integers. Then for $p= \omega(\Delta^{12}
n^{-1/\min\{2d,\Delta\}} \log^3 n)$, a graph $G \sim \gnp$ is w.h.p.\
$\mathcal H(n,\Delta,d)$-universal.
\end{theorem}

To prove this theorem, it will be sufficient to prove that it holds for $p= \omega(\Delta^{12}
n^{-1/(2d)} \log^3 n)$ as it follows from  \cite{DKRR} for the other minimum.

Next, let $\mathcal{H}(n,\Delta, d, g)\subseteq\mathcal H(n,\Delta,d)$ denote the family of graphs which additionally have girth at least $g$ (the girth of a graph is the length of its shortest cycle). In our second main result we further restrict ourselves to graphs with girth at least 7, where we obtain better bounds for $p$.

\begin{theorem}\label{thm:girthuniversality}
Let $n$ be a positive integer, and let $d=d(n)$ and $\Delta = \Delta(n) > 1$ be integers. Then for $p= \omega(\Delta^{12} n^{-1/d}\log^3 n)$, a graph $G \sim \gnp$ is w.h.p.\ $\mathcal H(n, \Delta, d, 7)$-universal.
\end{theorem}

Another example of a family of graphs which has attracted the attention of
various researchers is the family of \emph{bounded degree} trees.
Let $\mathcal T(n,\Delta)$ be the family of all forests on $n$
vertices with maximum degree bounded by $\Delta$. Alon, Krivelevich
and Sudakov showed in \cite{AKS} that for fixed $\Delta>0$ and
$0<\varepsilon<1$, there exists a constant $c=c(\Delta,\varepsilon)$
such that a typical member of $\mathcal G(n,c/n)$ is $\mathcal
T((1-\varepsilon)n,\Delta)$-universal. The constant $c$ in this result was further improved in \cite{BCPS}.
Later on, Balogh, Csaba and
Samotij showed in \cite{BCS} that $\mathcal G(n,c/n)$ is w.h.p.\ (with high
probability) $\mathcal T((1-\varepsilon)n,\Delta)$-universal
even if an adversary is allowed to delete at most (roughly) half of
the edges touching any vertex. Note that universality for spanning trees can not be true for $p=c/n$, as at such a low density the random graph is w.h.p.\ disconnected.
As it turns out, results for spanning subgraphs are much harder to obtain.
In the case of the family of \emph{spanning trees} $\mathcal T(n,\Delta)$,
the best bound known for $\gnp$ to be $\mathcal
T(n,\Delta)$-universal is $p=\omega(\Delta n^{-1/3}\log^2n)$, due to
Johannsen, Krivelevich and Samotij \cite{JKS}.
The following immediate corollary of Theorem~\ref{thm:girthuniversality} improves this bound to $p=\omega(\Delta^{12} n^{-1/2}\log^3 n)$.

\begin{corollary}\label{thm:treesuniversality}
Let $n$ be a positive integer, and let $\Delta = \Delta(n) > 1$ be an integer. Then for $p= \omega(\Delta^{12} n^{-1/2}\log^3 n)$, a graph $G \sim \gnp$ is w.h.p.\ $\mathcal T(n,\Delta)$-universal.
\end{corollary}

The proofs of Theorems~\ref{main} and ~\ref{thm:girthuniversality} use simple
partitioning lemmas for graphs and an
embedding technique based on matchings, developed by Alon and
F\"uredi in \cite{AF} and by Ruci\'nski in \cite{R}. Using similar
technique, we also managed to obtain a general embedding result in a model of random graphs where each edge is being colored uniformly at random in one color from a given set of colors. This leads us to the second part of our paper.

Let $G\sim \gnp$ and assume that each edge of $G$ is
colored uniformly at random with one of the colors
from the set $[c]:=\{1,\ldots,c\}$. This model is referred to as
$\gnpk$. For a given graph $H$ we say that a typical member of
$G\sim \gnpk$ contains a \emph{rainbow copy} of $H$, if $G$ contains
as a subgraph a copy of $H$ with all the edges colored in distinct
colors. In \cite{FL}, Frieze and Loh showed that for $p\geq
(1+\varepsilon)\log n/n$ and $c=n+o(n)$, a typical member of
$\gnpk$ contains a rainbow Hamilton cycle. Note that their result is
asymptotically optimal in both $p$ and the number of colors $c$. In
the following theorem we provide bounds on the edge probability $p$ (do not believed to be optimal), for which given any graph $H$ on $n$ vertices with $\Delta(H)=O(1)$, one can find a rainbow copy of $H$ in a typical member of $\gnpk$, provided $c=(1+o(1))|E(H)|$ ($c$ is asymptotically optimal).

%For a constant $\Delta$, $d\leq \Delta$ and $H\in \mathcal H(n,\Delta,d)$, if
%we randomly color the edges of a typical member of $\gnp$ with
%$(1+o(1))|E(H)|$ colors, then one can find a \emph{rainbow} copy of
%$H$ (that is, a copy of $H$ with all the edges colored in distinct
%colors).
\begin{theorem}\label{rainbow}
Let $\alpha>0$, let $\Delta$ and $d$ be integers, let $n$ be a sufficiently large integer and let $H\in\mathcal
H(n,\Delta,d)$. Then $G\sim \gnpk$ w.h.p.\
contains a rainbow copy of $H$, provided that $p \geq
n^{-1/d} \log^{5/d} n $ and $c = (1 + \alpha) |E(H)|$.
\end{theorem}

We remark that all of our proofs might be easily improved in terms of $\log n$ and $\Delta$ factors. Since we believe that our bounds are far from being optimal, we did no effort in optimizing those factors.

\textbf{Notation.} Our graph-theoretic notation is standard and
follows that of \cite{West}. For a graph $G$, let $V=V(G)$ and
$E=E(G)$ denote its sets of vertices and edges, respectively. For
subsets $U,W \subseteq V$, and for a vertex $v \in V$, we denote by
$E_G(U)$ all the edges of $G$ with both endpoints in $U$, by
$E_G(U,W)$ all the edges of $G$ with one endpoint in $U$ and one
endpoint in $W$ and by $E_G(v,U)$ all the edges with one endpoint
being $v$ and one endpoint in $U$. We write $N_G(v)$ for the
neighborhood of $v$ in $G$ and $\deg_G(v)$ for its degree. Moreover, we write $N_G(U)$ for the neighborhood of a set $U\subseteq V$.
For any positive integer $k$ and every vertex $V$ we denote the following set as \emph{$k$-neighborhood} of $v$:
$$\left\{v\in V \mid \text{the distance between $u$ and $v$ is at most $k$} \right\} $$
We say that a set $S\subseteq V$ is \emph{$k$-independent} if and only if (in $G$) the distance between any two vertices of $S$ is at least $k+1$.

Given a graph $G$ and a positive constant $d>0$ we denote by
$D_{d}(G)$ the set of all vertices of $G$ with degree exactly $d$,
by $D_{\leq d}(G)$ the set of all vertices of degree at most $d$ and
in a similar way we define $D_{<d}(G)$, $D_{>d}(G)$ and $D_{\geq
d}(G)$. When it is clear to which graph $G$ we refer, we just denote
it by $D_d$, $D_{\leq d}$ etc.

Given two graphs $H$ and $G$, a bijection $f$ from $V(H)$ to $V(G)$
is called an \emph{embedding} of $H$ to $G$ if it maps each edge of
$H$ to an edge of $G$. In case that one assigns colors to the edges
of $G$, an embedding $f$ of $H$ to $G$ is called a \emph{rainbow
embedding} if in addition it maps the edges of $H$ into edges with
distinct colors in $G$.

Throughout the paper, wherever we use $\log n$ we refer to the
natural logarithm.

\section{Preliminaries}

\subsection{Probabilistic Tools}

We will need to employ bounds on large deviations of random
variables. We will mostly use the following well-known bound on the
lower and the upper tails of the binomial distribution due to
Chernoff (see \cite{JLR}).

\begin{lemma}\label{Che}
If $X \sim \emph{\text{Bin}}(n,p)$, then
\begin{itemize}
\item  $\Pr\left(X<(1-a)np\right)<e^{-a^2np/2}$ for every $a>0$;
\item $\Pr\left(X>(1+a)np\right)<e^{-a^2np/3}$ for every $0<a<3/2.$
\end{itemize}
\end{lemma}
The proof of the following slightly more general bounds follows directly from the Chernoff bound and is left as an exercise for the reader (see for example Problem~1.7 in \cite{dubhashi2009concentration}).

\begin{lemma}
\label{thm:chernoff}
Let $p, q\in [0,1]$ and let $X_1, \dots, X_n\in \{0,1\}$ be $n$ indicator variables and $X:=\sum_{i=1}^nX_i$. If for each $1\le i \le n$
$$\mathbb{E}[X_i|X_1, \dots, X_{i-1}] \geq p \quad \text{ and } \quad \mathbb{E}[X_i|X_1, \dots, X_{i-1}] \le q, $$
then it holds for every $0<\alpha<1$ that
$$\Pr[X\geq (1+\alpha)nq]\le e^{-\alpha^2nq/3} \quad \text{ and } \quad \Pr[X\le (1-\alpha)np]\le e^{-\alpha^2np/2}.$$
\end{lemma}

\subsection{Graph-Theoretic Facts}
\label{sec:uni_graph}
In this section we mention a few facts about graphs which are used
extensively throughout the paper.

The first two lemmas consider the existence of $k$-independent sets in a graph.
\begin{lemma}\label{IndependentSetInGraphWithBoundedDegree}
Let $G$ be a graph on $n$ vertices with maximum degree $\Delta\geq
2$ and let $S\subseteq V(G)$ be such that the maximum degree of all
vertices in $S$ is at most $d$ (where $d\geq 1$). Then, $S$ contains
a set $U\subseteq S$ of size at least
$\frac{|S|}{d\Delta^{k}}$ which is $k$-independent in $G$.
\end{lemma}

\begin{proof}  Build $U$ greedily as follows: start
with $L:=S$ and $U:=\emptyset$. In each step add an arbitrary vertex $v\in L$ to $U$
and delete the $k$-neighborhood of $v$ (including $v$ itself) from
$L$. Since after each addition of a vertex to $U$ we delete at most

$$1+d+d(\Delta-1)+\ldots+d(\Delta-1)^{k-1}\leq d\Delta^k$$
vertices from $L$, we obtain the required.
\end{proof}

%\begin{lemma}\label{NumberOfVerticesOfDegreeAtMostD}
%Let $n$ and $d$ be positive integers. Then for
%every $H\in \mathcal H(n,d)$ we have that $|D_{\leq
%d}(H)|\geq \frac{n}{d+1}.$
%\end{lemma}
%
%\begin{proof} Let $H\in \mathcal H(n,d)$. Now, using the fact that $|D_{>d}|=n-|D_{\leq d}|$, we obtain that
%$$dn\geq \sum_{v\in V(H)}d_H(v)\geq 0\cdot |D_{\leq d}|+(d+1) \cdot
%(n-|D_{\leq d}|).$$
%
%Therefore, we conclude that $|D_{\leq d}|\geq
%\frac{n}{d+1}$ as
%required.
%\end{proof}

%I DON'T THINK THAT WE NEED THIS LEMMA
%\begin{lemma}\label{IndependentSetMadeOfSmallVertices1}
%Let $n$ be an integer, let $\Delta$ and $d$ be positive constants
%and let $H\in\mathcal H(n,\Delta,d)$. Then, for every integer $k$,
%$H$ contains a $k$-independent set $U\subseteq H$ for which:
%\begin{enumerate}[$(i)$]
%\item $U\subseteq D_{\leq d}$, and
%\item $|U|\geq \frac{\left(1-d+\lfloor d\rfloor \right)n}{(d+1)(d\cdot\Delta^{k-1}+1)}$.
%\end{enumerate}
%\end{lemma}
%
%\begin{proof} First, it follows from Lemma \ref{NumberOfVerticesOfDegreeAtMostD} that $|D_{\leq d}(H)|\geq \frac{\left(1-d+\lfloor d\rfloor
%\right)n}{d+1}$. Second, note that for $H'=H^k[D_{\leq d}(H)]$ we
%have that $\Delta(H')\leq d\cdot \Delta^{k-1}$. Now, applying Lemma
%\ref{IndependentSetInGraphWithBoundedDegree} to $H'$ we conclude
%that there exists an independent set $U$ in $H'$ (which is trivially
%a $k$-independent set of $H$) of size at least
%$$|U|\geq \frac{|D_{\leq d}|}{d\cdot \Delta^{k-1}+1}\geq
%\frac{\left(1-d+\lfloor d\rfloor \right)n}{(d+1)(d\cdot
%\Delta^{k-1}+1)}$$ as required.
%\end{proof}

\begin{lemma}\label{IndependentSetMadeOfSmallVertices}
Let $G$ be a graph on $n$ vertices with maximum degree $\Delta \geq 2$ and let $d$ be an integer such that $dn \geq 2|E(G)|$. Then, for any integer $k$, $G$ contains a $k$-independent set $U \subseteq D_{\le d}(G)$ of size $|U|\geq \frac{n}{( d+1)d \Delta^{k}}$.
\end{lemma}

\begin{proof}
First, we claim that $|D_{\leq d}(G)|\geq \frac{n}{d+1}$. Indeed, let $G$ be a graph which satisfies the conditions of the lemma for some $\Delta$. Using the fact that $|D_{>d}|=n-|D_{\leq d}|$, we obtain that
$$dn \geq \sum_{v\in V(G)}\deg_G(v)\geq 0\cdot |D_{\leq d}|+(d+1) \cdot
(n-|D_{\leq d}|).$$
Therefore, we conclude that $|D_{\leq d}|\geq
\frac{n}{d+1}$.

Applying Lemma~\ref{IndependentSetInGraphWithBoundedDegree} we conclude
that there exists a $k$-independent set $U\subseteq D_{\le d}$ in $G$ of size at least
$$|U|\geq \frac{|D_{\leq  d}|}{ d \Delta^{k}}\geq
\frac{n}{( d+1)d \Delta^{k}},$$
as required.
\end{proof}

A graph $G$ is called \emph{$d$-degenerate} if every subgraph $G' \subseteq G$ contains a vertex of induced degree at most $d$. A moment's thought reveals that every graph $H \in \mathcal{H}(n, \Delta, d)$ is $d$-degenerate (but not vice versa). The following observation follows directly from the definition of $d$-degenerate graphs.

\begin{observation} \label{Degeneracy}
Let $n, \Delta$ and $d$ be positive integers and let $H$ be a
$d$-degenerate graph on $n$ vertices. Then there exists an ordering $(v_1, \ldots,
v_n)$ of the vertices of $H$ such that
$$|N(v_i) \cap \{v_1, \ldots,
v_{i-1}\}| \leq d$$
for every $2 \leq i \leq n$.
\end{observation}

\section{Partitioning Lemmas}
\label{sec:uni_partitioning}
In this section we prove some lemmas about partitioning graphs from
$\mathcal H(n,\Delta,d)$ and $\mathcal{T}(n, \Delta)$. Before that,
we define a class of graphs which can be partitioned in a ``nice''
way, and then we show that $\mathcal H(n,\Delta,d)$ and
$\mathcal{T}(n, \Delta)$ belong to this class for suitably chosen
parameters.

\begin{definition}\label{def1}
Let $n, d$ and $t$ be positive integers and let $\varepsilon$ be a
positive number. The family of graphs $\mathcal{F}(n, t,
\varepsilon, d)$ consists of all graphs $H$ on $n$ vertices for
which the following holds. There exists a partition $V(H) = W_0 \cup
\ldots \cup W_t$ such that:
\begin{enumerate}[(i)]
\item $|W_t| = \lfloor \varepsilon n \rfloor$,
\item $W_0 = N(W_t)$,
\item $W_t$ is $3$-independent,
\item $W_i$ is $2$-independent for every $1\leq i\leq t-1$, and
\item for every $1\leq i\leq t$ and for every $w\in W_i$, $w$ has at
most $d$ neighbors in $W_0\cup\ldots \cup W_{i-1}$.
\end{enumerate}
\end{definition}

Now, we show that $\mathcal{H}(n, \Delta, d) \subseteq
\mathcal{F}(n, 4\Delta^6 \log n+1, \varepsilon, 2d)$.

\begin{lemma} \label{mainlemma}
Let $n$ be a positive integer, let $\Delta = \Delta(n)\geq2$ and $d=d(n)\geq2$
be integers and let $\varepsilon_0 = 1/(4\Delta^6)$. Then for every
$\varepsilon \leq \varepsilon_0$ we have
$$\mathcal{H}(n, \Delta, d) \subseteq \mathcal{F}(n, 4\Delta^6 \log n+1, \varepsilon,
2d).$$
\end{lemma}
\begin{proof}
%For simplicity of calculations, we prove Lemma~\ref{mainlemma} for  $\varepsilon_0 = 1 / ((d +1)(d \cdot \Delta^3 + 1))$.
%Let $K = 1 / \varepsilon_0$. Note that $K \geq \frac{1}{- \log(1 - \varepsilon_0)} = \frac{1}{\log(1 / (1 - \varepsilon_0)}$, and therefore  $K\log n \geq -\log_{1-\varepsilon_0}n$ (we use the fact that $\log(1 - x) \le -x$ for every $0 < x < 1$). We prove that $\mathcal{H}(n, \Delta, d) \subseteq \mathcal{F}(n, t, \varepsilon, 2d)$, where $t = \lceil K \log n \rceil$ and $\varepsilon \leq \varepsilon_0$. Since $K \leq 4\Delta^5$, this easily implies the conclusion of Lemma~\ref{mainlemma} as we can always add empty $W_i$'s.
Let $H\in \mathcal H(n,\Delta,d)$ and $t = 4\Delta^6 \log n+1$. We show that $H \in \mathcal{F}(n, t, \varepsilon, 2d)$, for every $\varepsilon \leq \varepsilon_0$.

Using
Lemma~\ref{IndependentSetMadeOfSmallVertices}, one can find a
$4$-independent set $U \subseteq D_{\leq d}(H)$ of size
$$|U| \geq \frac{n}{(d + 1)d \Delta^4 }  \geq \varepsilon_0 n.$$
Let $W_t \subseteq U$ be an arbitrary subset of
size $\lfloor \varepsilon n \rfloor$, and set $W_0=N_H(W_t)$ and
$H_{t-1}:=H\setminus (W_0\cup W_t)$. We further partition $H_{i}$,
for $i = t - 1, \ldots, 1$, as follows:
\begin{itemize}
\item If $V(H_i) = \emptyset$ then set $W_i := \emptyset$ and $V(H_{i-1}) :=
\emptyset$.
\item Otherwise, $H_i \in \mathcal{H}(|H_i|, \Delta, d)$ and thus by Lemma~\ref{IndependentSetMadeOfSmallVertices} there exists a $2$-independent set $U \subseteq D_{\leq d}(H_i)$ of size $|U| \geq \frac{|H_i|}{(d + 1)d \cdot \Delta^2} \geq \frac{|H_i|}{2\Delta^4}\geq \varepsilon_0 |H_i|$. Set $W_i := U$ and $H_{i-1} := H_i \setminus W_i$.
\end{itemize}
Using the fact that $\log(1 - x) \le -x$ for every $0 < x < 1$, we have that
$$t = 4\Delta^6 \log n +1= \log n / \varepsilon_0+1 \geq - \log n / \log (1 - \varepsilon_0) +1 = -\log_{1 - \varepsilon_0} n +1.$$
Since for each $i$ we have that $|V(H_i)|\leq
(1-\varepsilon_0)|V(H_{i+1})|$, and since $t\geq
-\log_{1-\varepsilon_0} n+1$, it follows that $|V(H_1)| \le 1$.

Now, let $V(H) = W_0 \cup \ldots \cup W_t$ be the obtained partition
and note that each vertex $w \in W_i$ has at most $d$ neighbors in
$W_1 \cup \ldots \cup W_{i-1}$ for $2 \leq i < t$ (it follows
immediately from the construction). Since all the properties
$(i)-(iv)$ of Definition \ref{def1} follow easily from the
construction, it thus remains to show that Property $(v)$ holds.
That is, we need to show that every vertex in $w \in W_1\cup\ldots
\cup W_{t-1}$ has at most $d$ neighbors in $W_0$, and then we
conclude that every vertex in $W_1\cup\ldots\cup W_t$ sends at most
$2d$ ``back-edges''. For this aim, note first that every vertex in
$W_t$ has at most $d$ neighbors in $W_0$, and that $W_0 = N_H(W_t)$.
Therefore, if there exists a vertex $w \in W_1\cup\ldots\cup
W_{t-1}$ with at least $d + 1$ neighbors in $W_0$, then there must
exist at least two vertices $x, y \in W_t$ such that $N_H(x) \cap
N_H(w) \neq \emptyset$ and $N_H(y) \cap N_H(w) \neq \emptyset$.
Therefore, one can find a path of length four between $x$ and $y$,
which clearly contradicts the assumption that $W_t$ is
$4$-independent. This completes the proof.
\end{proof}

Next, we show that $\mathcal{H}(n, \Delta, d, 7) \subseteq \mathcal{F}(n,
16 d^2 \Delta^2 \log n+1, \varepsilon, d)$.

\begin{lemma} \label{mainlemmaTrees}
Let $n$ be a positive integer, let $\Delta = \Delta(n)$ and $d=d(n)\geq 2$ be integers, and let $\varepsilon_0 = 1/(2d^2\Delta^6)$. Then for every $\varepsilon \leq \varepsilon_0$ we have
$$\mathcal{H}(n, \Delta, d, 7) \subseteq \mathcal{F}(n, 16 d^2 \Delta^2 \log n+1, \varepsilon, d).$$
\end{lemma}
\begin{proof}
Let $\gamma = \frac{1}{8 (d+1)(d-1) \Delta^2}\le \frac{1}{d^2}$ and observe that for $d\geq 2$
\begin{equation} \label{eq:gamma}
\frac{1 - (d+1)(d-1)\Delta^2\gamma}{(d+1)(d-1)\Delta^2} > \gamma \quad \text{and}
\quad \frac{d^2(1 - (d+1)(d-1)\Delta^2\gamma)}{d+1} > 1.
\end{equation}
%Furthermore, set
%$$ \varepsilon_0 = \frac{1}{3(2\Delta^5 + 1)} \quad \text{and} \quad K = 1 / \gamma.$$
%Note that $K \geq \frac{1}{\log(1 / (1 - \gamma))}$, and therefore $K \log n \geq - \log_{1-\gamma} n$. We prove that $\mathcal{T}(n, \Delta) \subseteq \mathcal{F}(n, t, \varepsilon, 2)$, where $t = \lceil K \log n \rceil$ and $\varepsilon \leq \varepsilon_0$. Since $K \leq 36 \Delta$, Lemma~\ref{mainlemmaTrees} easily follows.
Let $H\in \mathcal H(n,\Delta, d, 7)$ and $t = 16d^2\Delta^2 \log n+1$. We show that $H \in \mathcal{F}(n, t, \varepsilon, d)$, for every $\varepsilon \leq \varepsilon_0$.

Using Lemma~\ref{IndependentSetMadeOfSmallVertices}, we find a
$6$-independent set $U \subseteq D_{\leq d}(H)$ of size
$$|U| \geq \frac{n}{d(d+1)\Delta^6} \geq \varepsilon_0 n.$$
For a fixed $\varepsilon \leq \varepsilon_0$, let $W_t \subseteq U$
be an arbitrary subset of size $\lfloor \varepsilon n \rfloor$, and
set $W_0=N_H(W_t)$, $X=N_H(W_0)$, and $H_{t-1}:=H\setminus (W_0 \cup
W_t)$.

%Note that the only problematic property is $(vii)$ and we should
%treat it carefully. Theoretically, one would expect to apply Lemma
%\ref{PartitionWith2independentSets} to $H_1$ and get the desired
%partition. Practically, some problems might occur if in the
%partition $W_1\cup\ldots \cup W_{t-1}$ there exists an index $i$ and
%a vertex $x\in X$ such that $x\in W_i$ and $x$ has $\lceil d\rceil$
%neighbors in $W_1\cup\ldots \cup W_{i-1}$. This causes problems
%since $X=N_H(W_0)$ and therefore, by adding $W_0$ to this partition,
%$x$ has at least one more neighbor in $W_0\cup \ldots \cup W_{i-1}$
%and it violates $(vii)$. In order to overcome this difficulty we
%need to be more careful in the way we choose the $W_i$'s.

For $i=t-1,\dots, 1$, we iteratively find subsets of vertices $W_{i}
\subseteq V(H_i)$ (and set $H_{i-1}:=H_i \setminus W_{i}$), in such
a way that at the end of the process the obtained partition $V(H) =
W_0 \cup \ldots \cup W_t$ satisfies Properties $(i)-(v)$ of
Definition \ref{def1}.

If $V(H_i) = \emptyset$ then set $W_i := \emptyset$ and $V(H_{i-1}) :=
\emptyset$. Otherwise, construct $W_{i}$ as follows:
\begin{enumerate}[(1)]
\item If there exists a $2$-independent set $U\subseteq D_{\le d-1}(H_i)$ of size $U \geq \gamma |V(H_i)|$, then set $W_i:=U$.

\item Otherwise, pick a $2$-independent set $W_{i} \subseteq D_{\leq d}(H_i) \setminus X$ of size $|W_{i}| \geq \gamma |V(H_i)|$.
\end{enumerate}
Observe that a vertex can have at most one neighbor in $W_0$. Otherwise, we would either have that $W_t$ is not $6$-independent or that there exists a cycle of length $4$ in $H$, both yielding a contradiction. Therefore, by the definition of $W_i$, we ensure that Property $(v)$ of Definition \ref{def1} is satisfied.
We now claim that whenever $(1)$ fails, there exists a $2$-independent
set $U\subseteq D_{\leq d}(H_i)$ of size $|U| \geq \gamma n_i$
(where $n_i=|V(H_i)|$) such that $U\cap X=\emptyset$ as required in
$(2)$. We remark that we always consider the graph $H_i$ when we write $D_{d}$, $D_{\le d}$ or $D_{\geq d}$ in the following calculations.

To prove our claim, suppose that there is no
$2$-independent set $U\subseteq D_{\le d-1}(H_i)$ of size at least
$\gamma n_i$.
First, note that by
Lemma~\ref{IndependentSetInGraphWithBoundedDegree} we have $$|D_{\le
d-1}| \leq(d-1)\Delta^2\gamma n_i.$$ Second, since $H_i\in \mathcal
H(n_i,\Delta, d, 7)$, it follows that
$$d n_i\geq \sum_{v\in V(H_i)}\deg_{H_i}(v)\geq 0\cdot |D_{\le d-1}| +\left(|D_{\le d}|-|D_{\le d-1}|\right)\cdot d +\left(n_i-|D_{\le d}|\right)\cdot (d+1),$$
and therefore
$$
|D_{\le d}| \geq n_i - |D_{\le d-1}|\cdot d.
$$
Using the bound on $|D_{\le d-1}|$, we get that
\begin{equation} \label{eq:D_d_large}
|D_{d}| = |D_{\le d}| - |D_{\le d-1}| \ge n_i - (d+1)|D_{\le d - 1}| \ge n_i \cdot (1 -
(d+1)(d-1)\Delta^2\gamma).
\end{equation}
Next, note that if $|X \cap D_{d}| \le d |D_{d}| / (d+1)$, then by
Lemma~\ref{IndependentSetInGraphWithBoundedDegree} there exists a
$2$-independent set $W_i \subseteq D_{d} \setminus X$ of size at
least
$$ |W_i| \geq \frac{|D_{d} \setminus X|}{d\Delta^2} \stackrel{\eqref{eq:D_d_large}}{\ge}  \frac{1 - (d+1)(d-1)\Delta^2\gamma}{d(d+1)\Delta^2} n_i \stackrel{\eqref{eq:gamma}}{\ge} \gamma n_i, $$
as required. Therefore, assume that $|X \cap D_{d}| > d|D_{d}| / (d+1)$.
Observe that $X$ is a $2$-independent set in $H_i$, as every vertex
in $X$ is a neighbor of a vertex in $W_0=N_H(W_t)$, $W_t$ is
$6$-independent and there are no cycles of length at most 6 in $H$. It thus follows that
$N_{H_i}(X) \cap X = \emptyset$ and every vertex in $N_{H_i}(X)$ has
exactly one neighbor in $X$. Therefore,
$$ |N_{H_i}(X)| \geq d |X \cap D_{d}| > d^2 |D_{d}| / (d+1),$$
and it follows from \eqref{eq:D_d_large} that
$$ |N_{H_i}(X)| > \frac{d^2}{d+1} \cdot  (1 - (d+1)(d-1)\Delta^2\gamma)n_i \stackrel{\eqref{eq:gamma}}{>} n_i,$$
which is not possible. Hence, one can always find a $2$-independent
set $W_i \subseteq V(H_i)$ of size at least $\gamma n_i$ as
required.

Using the fact that $\log(1 - x) \le -x$ for every $0 < x < 1$, we have that
$$t+1 = 16d^2 \Delta^2 \log n+1 \geq \log n / \gamma +1\geq - \log n / \log (1 - \gamma)+1 = -\log_{1 - \gamma} n+1.$$
Since for each $i$ we have that $|V(H_i)|\leq
(1-\gamma)|V(H_{i+1})|$, and since $t\geq
-\log_{1-\gamma} n$+1, it follows that $|V(H_1)|\le 1$. Finally, let $V(H)=W_0 \cup \ldots
\cup W_t$ be the obtained partition. It follows immediately from the
construction that Properties $(i)-(v)$ of Definition \ref{def1}
hold. This completes the proof.
\end{proof}

\section{Proof of Theorem~\ref{main} and Theorem~\ref{thm:girthuniversality}}\label{sec:universality}
\label{sec:uni_proof1}
In this section we prove Theorem~\ref{main} and Theorem~\ref{thm:girthuniversality}. These theorems follow easily from the following theorem and Lemma~\ref{mainlemmaTrees} and~\ref{mainlemma}.

\begin{theorem}\label{universality_F}
Let $n$ and $t$ be positive integers, let $d=d(n)\geq2$ be an integer, and let $\varepsilon < \frac{1}{2d}$. Then, a
graph $G \sim \gnp$ is w.h.p.\ $\mathcal{F}(n, t, \varepsilon, d)$-universal, provided that $p = \omega\left( \varepsilon^{-1} t n^{-1/d} \log^2 n\right)$.
\end{theorem}

In order to prove Theorem~\ref{universality_F}, we use a similar embedding
algorithm as the one presented in \cite{KL} (and previously in \cite{DKRR}). Let $d$ be a positive integer and $\varepsilon$ be a positive constant.
%For any $n, \Delta$ and $d$, let $t, C_T$ and $\gamma$ be the values given by Lemma~\ref{mainlemma} and let
%$$ \delta:=\frac{1}{2} \quad \text{ and } \quad \varepsilon:=\min \left\{\gamma, \frac{2}{6d^2+1} \right\}.$$
Our goal is to show that, whenever a graph $G$ is ``good'' with
respect to some properties, then $G$ is $\mathcal{F}(n, t, \varepsilon, d)$-universal.

Before we state formally what a ``good'' graph is, we define the
following auxiliary bipartite graph. For a graph $G$, an integer
$k$, a subset $U\subseteq V(G)$ and a collection $\mathcal{L}$ of
pairwise disjoint $k$-subsets of $V(G) \setminus U$, define the
bipartite graph $\mathcal{B}(\mathcal{L}, U)$ as follows: the parts
are $\mathcal{L}$ and $U$, and two elements $L\in \mathcal{L}$ and
$u \in U$ are adjacent if and only if $L \subseteq N_G(u)$. Now we
can define the notion of an \emph{$(n, t, \varepsilon, d)$-good} graph $G$.

\begin{definition} \label{def:good_graph}
A graph $G$ on $n$ vertices is called \emph{$(n, t, \varepsilon, d)$-good} if
there exists a partition $V(G) = V_0 \cup V_1 \cup \dots \cup
V_t$ with $$|V_i|=\frac{\varepsilon n}{16t} \quad \text{ for } 1\le i\le t\quad \text{ and } \quad |V_0|=(1-\frac{\varepsilon}{16})n, $$
such that for $p \geq \varepsilon^{-1} t n^{-1/d} \log^2 n$ the following properties
hold.
\begin{itemize}
\item[(P1)] There exists a set $\mathcal{K}\subset V_0$ of $\varepsilon n$ vertex-disjoint $d$-cliques such that for all $U\subset V(G) \setminus V(\mathcal{K})$ with $|U|\le (p/2)^{-d} / 2$, we have
$$\left|\left\{K_d\in \mathcal{K} \mid V(K_d)\subset N_G(u) \text{ for some } u\in U \right\} \right|\geq \frac{1}{2^{d+2}}p^{d} |U| \varepsilon n.$$
\item[(P2)] Let $1\le k \le d$, and $\mathcal{L}$ be a collection of pairwise disjoint $k$-subsets of $V(G)$.

If $|\mathcal{L}|\le (p/2)^{-k} / 2$, then for each $i=1, \dots, t$
with $V_i\cap(\cup_{L\in\mathcal{L}}L)=\emptyset$, we have that
\begin{equation}
|N_{\mathcal{B}(\mathcal{L}, V_i)}(\mathcal L)|\geq
(p/2)^{k}|\mathcal{L}||V_i| / 2.
\end{equation}
If $|\mathcal{L}|\geq (p/2)^{-k} \log^{2(d-1)} n $, then for all $U$ with
$|U|\geq (p/2)^{-k} \log^{2(d-1)} n $ and $U\cap(\cup_{L\in
\mathcal{L}}L)=\emptyset$, the graph $\mathcal{B}(\mathcal{L}, U)$
has at least one edge.
\end{itemize}
\end{definition}

We first show that a random graph is typically good.
\begin{lemma}
\label{lem:random_graph_is_good} Let $\varepsilon <\frac{1}{2d}$ and let $n$ be a positive integer. Then, a graph $G \sim \gnp$ is w.h.p.\
$(n, t, \varepsilon, d)$-good, provided that $p=\omega\left( \varepsilon^{-1} t n^{-1/d} \log^2 n\right)$.
\end{lemma}
\begin{proof}
 Let $\varepsilon \leq \frac{1}{2d}$, let $p =\omega\left( \varepsilon^{-1} t n^{-1/d} \log^2 n\right)$ and let $G \sim \gnp$. Furthermore, let $q\geq p/2$ be such that $1-p=(1-q)^2$, and note that one can expose $G\sim\gnp$ as $G=G_1\cup G_2$, where $G_1$
and $G_2$ are two graphs sampled from $\mathcal G(n,q)$ independently (for more details we refer the reader to \cite{JLR}). We use $G_1$
to find a family of vertex-disjoint $d$-cliques, and then $G_2$ to
ensure the properties (P1) and (P2). For a simpler presentation, we assume from now on that $q$ is exactly $p/2$.

%First, observe that it follows by  that
%there exists a positive constant $\varepsilon_0 = \varepsilon_0(d)$
%such that $G_1$ contains at least $\varepsilon_0 n$ vertex-disjoint
%copies of $d$-cliques. Given a constant $\varepsilon \le
%\varepsilon_0$, let $\mathcal{K}$ be an arbitrary subset of $\lfloor
%\varepsilon n \rfloor$ such cliques.

First, expose the edges of $G_1$. Since
$$q=\omega\left(n^{-2/d}(\log n)^{1/\binom{d}{2}}\right),$$
it follows from \cite{johansson2008factors} that $G_1$ contains w.h.p.\ $\lfloor n/d\rfloor$ disjoint $d$-cliques. Let $\mathcal{K}$ be a family of $\varepsilon n$ vertex-disjoint $d$-cliques.  Next, fix an arbitrary
partition $V(G) = V_0 \cup \ldots \cup V_t$ as in
Definition~\ref{def:good_graph}, such that $V(\mathcal{K})\subset
V_0$. Finally, expose $G_2$. We now show that w.h.p.\ this partition satisfies
Properties (P1) and (P2).

%It follows by Lemma~3.1 in \cite{DKRR} that the number of $d$-cliques in $G_1$ is w.h.p.\ at least $(1-\delta)\binom{V_0}{d}(p/2)^{\binom{d}{2}}$ and that every vertex in $V_0$ is contained in at most $(1+\delta)(p/2)^{\binom{d}{2}}\frac{d}{|V_0|}\binom{|V_0|}{d}$ of these cliques. We can therefore find greedily a set $\mathcal{K}'$ of at least $\frac{1-\delta}{1+\delta}\frac{|V_0|}{d^2}\geq \varepsilon n$ disjoint $d$-cliques in $G_1[V_0]$. Let $\mathcal{K}\subseteq \mathcal{K}'$ be an arbitrary subset of cardinality $\varepsilon n$.

For $U\subseteq V(G) \setminus V(\mathcal{K})$ with $|U| \le
(p/2)^{-d}/2$, let
$$X(U):=|\left\{K_d\in \mathcal{K}\middle| K_d\subset N_{G_2}(u) \; \text{for some} \; u\in U \right\}|.$$
Note that $X(U)$ is the sum of i.i.d. indicator random variables
$X_{L}$ ($L\in \mathcal{K}$), such that $X_{L}=1$ iff $L\subseteq
N_{G_2}(u)$ for some $u\in U$. Since $|U|\le (p/2)^{-d}/2$, we have
that for each $L\in\mathcal{K}$,
$$\Pr[X_{L}=0]=(1-(p/2)^d)^{|U|}\le 1-\frac{|U|p^d}{2^d} +\frac{|U|^2p^{2d}}{2^{2d}} \le 1-\frac{|U|p^d}{2^d}(1-1/2) = 1-\frac{|U|p^d}{2\cdot2^d}.$$
(For the first inequality we use the fact that $(1-a)^b\leq
1-ab+(ab)^2$ for any positive integer $b$ and $0<a<1$).

Therefore, we have that $\Pr[X_{L}=1]\geq 2^{-d-1}|U|p^d$, which
implies that
$$\mathbb{E}[X(U)]\geq  2^{-d-1}|U|p^d |\mathcal{K}|\geq\frac{\log^{2d} n}{2^{d+1}}|U|.$$
Using Chernoff's bound we obtain that
$$\Pr\left[X(U)< \frac{p^d}{2^{d+2}}|\mathcal{K}||U|\right] \le 2e^{-\frac{\log^{2d} n}{8 \cdot 2^{d+1}} |U|}\le \frac{2}{n^{3|U|}}. $$
(The last inequality holds since $d>1$).

We can therefore upper bound the probability that there exists a set
$U$ that violates $(P1)$ by the following union bound
$$\sum_{\ell=1}^n \binom{n}{\ell}\frac{2}{n^{3\ell}}=o(1).$$

For property $(P2)$ we first assume that $|\mathcal{L}| \le
(p/2)^{-k} / 2$. Note that $X(\mathcal{L},
V_i):=|N_{\mathcal{B}(\mathcal{L}, V_i)}(\mathcal{L})|$ is the sum
of i.i.d. indicator random variables $X_v$ (for $v \in V_i$), where
$X_v=1$ iff $L \subset N_{G_2}(v)$ for some $L\in \mathcal{L}$.
Since $(p/2)^k|\mathcal{L}| \le 1/2$, using the fact that
$(1-a)^b\leq 1-ab+ (ab)^2/2$ holds for every integer $b$ and any
positive constant $a$ for which $ab<1$ (follows from the binomial
formula), we observe that
$$\mathbb{E}[X(\mathcal{L}, V_i)] \geq|V_i|\left( 1-(1-(p/2)^k )^{|\mathcal{L}|} \right)\geq (1-1/4)(p/2)^k|\mathcal{L}||V_i|.$$
Using Chernoff's bound we obtain that
$$Pr[X(\mathcal{L}, v_i)< (p/2)^k|\mathcal{L}||V_i| / 2]\le \exp\left[- \mathbb{E}[X(\mathcal{L},v_i)] / 36\right]\le \frac{1}{n^{3d|\mathcal{L}|}}, $$
where the last inequality follows as
$$(p/2)^k |V_i|\geq (p/2)^d \cdot \frac{\varepsilon n}{16 t}\geq \frac{n\log^{2d} n}{2^{d+4} n}=\omega(d\log n).$$
Thus, the probability for having sets $\mathcal{L}$ and $V_i$ such that $|N_{\mathcal{B}(\mathcal{L},
V_i)}(\mathcal{L})|< (p/2)^{k} |\mathcal{L}||V_i| / 2$ can be
bounded by
$$t\sum_{\ell=1}^n \binom{\binom{n}{k}}{\ell}\frac{1}{n^{3d\ell}}=o(1).$$

Next, assume that $|\mathcal{L}|\geq(p/2)^{-k} \log^{2(d-1)} n$. Observe that each edge
in $\mathcal{B}(\mathcal{L}, U)$ is present with probability
$(p/2)^k$, hence the probability that there are no edges is bounded
by
$$ (1 - (p/2)^k)^{|\mathcal{L}| |U|} \leq \exp\left[- (p/2)^k \cdot |\mathcal{L}| |U|\right].$$
 %Note that the number $Y(\mathcal{L}, U)$ of edges in $\mathcal{B}(\mathcal{L},U)$ is the sum of i.i.d. random variables $Y_{L, u}$ for $u\in U$ and $L\in \mathcal{L}$ that are $1$ if $L\subseteq N_{G_2}(u)$ and $0$ otherwise. Since we have $\mathbb{E}[Y(\mathcal{L}, U)]=(p/2)^k|\mathcal{L}||U|$, it follows by Chernoff's bound that
%$$Pr[Y(\mathcal{L}, U)=0]\le 2\exp\left[-\frac{1}{12}p^k|\mathcal{L}||U|\right] %$$
Furthermore, for $r,\ell\geq (p/2)^{-k} \log^{2(d-1)} n$, the number of
collections of $k$-subsets $\mathcal{L}$ with $|\mathcal{L}|=\ell$
is at most $n^{k\ell}$, and the number of subsets $U$ with $|U|=r$
is at most $n^r$. We thus have that
$$\Pr[\exists \mathcal{L}, U \text{ with } |\mathcal{L}|=\ell, |U|=r \text{ and } e(\mathcal{B}(\mathcal{L},U))=0]\le \exp\left[(k\ell+r)\log n - (p/2)^k\ell r\right].$$
Note that
$$(k\ell +r)\log n \le k \cdot (\ell \log n + r \log n) \le 2 k \cdot \frac{r \ell (p/2)^k}{\log^{2d-3} n} \le (p/2)^k \ell r / 2$$
for $n$ large enough, and hence,
$$\exp\left[(k\ell+r)\log n - (p/2)^k\ell r\right]\le \exp\left[-(p/2)^k \ell r / 2\right]\le \exp\left[- (p/2)^{-k} \log n / 2\right]= o(1).$$
We therefore conclude that the probability for the existence of such
sets $\mathcal{L}$ and $U$ without an edge is $o(1).$
\end{proof}

Now we want to show that any $(n, t, \varepsilon, d)$-good graph is
$\mathcal{F}(n, t, \varepsilon, d)$-universal. Let $G$ be a a $(n, t, \varepsilon, d)$-good graph with
a partition $V(G) = V_0 \cup \dots \cup V_t$ and a clique-set
$\mathcal{K}$. We construct an embedding $f:V(H) \rightarrow V(G)$
for a given graph $H\in \mathcal{F}(n, t, \varepsilon, d)$ as follows.

Let $H=W_0 \cup \dots \cup W_{t}$ be the partition of $H$ that
satisfies the conditions $(i)-(v)$ of Definition~\ref{def:good_graph}. For every $v \in W_t$ let $L(v):=N_G(v)\cap W_0$
denote the neighborhood of $v$ in $W_0$.
Note that since $W_t$ is $3$-independent, we have that  $L(u)\cap L(v)=\emptyset$ for $u\ne v$.
In a first step we choose
an arbitrary injective mapping $f_0 : W_0 \rightarrow
V(\mathcal{K})$ such that for every $w\in W_t$ the vertices in
$L(w)$ all map to vertices of the same clique in $\mathcal{K}$. Such a mapping exists as $\mathcal{K}$ consists of $\lfloor
\varepsilon n \rfloor$ $d$-cliques and there are exactly that many
sets $L(w)$, each of which contains at most $d$ vertices. Moreover, such a mapping is valid as there can not be edges between $L(u)$ and $L(w)$ for $u\ne w$ (because $W_t$ is $3$-independent).

For $i=1, \dots, t$, we iteratively construct $f_i:(W_0\cup\dots\cup
W_i) \rightarrow (V_0 \cup \dots \cup V_i)$ from $f_{i-1}$ as
follows. Let $V^*_i:=(V_0\cup\dots\cup V_{i})\setminus
\textrm{Img}(f_{i-1})$. We want to embed $W_i$ to $V_i^*$. For $w\in W_i$
let $L_i(w):=f_{i-1}(N_H(w)\cap(\cup_{j=0}^{i-1}W_j))$ and let
$\mathcal{L}_{i}:=\left\{L_i(w) \mid w\in W_i\right\}.$ Here it is
crucial that $W_i$ is $2$-independent and therefore $L_i(w)\cap
L_i(w')=\emptyset$ for $w\ne w'\in W_i$. Since a vertex $w\in W_i$
can be mapped only to the vertices in
$$\{v\in V_i^* \mid L_i(w) \subseteq N_G(v)\},$$
we can extend $f_{i-1}$ by a $\mathcal{L}_i$ matching in $B_i:=\mathcal{B}(\mathcal{L}_i, V_i^*)$
(recall that in $B_i$ the set $L_i(w)\in \mathcal{L}_i$ is connected to a vertex $v\in V_i^*$
 if and only if $L_i(w)\subseteq N_G(v)$).
 More precisely, for a matching $\mathcal{M}$ which saturates $\mathcal L_i$ (an $\mathcal L_i$-matching), we define $f_i$ as follows:
 For $w\in W_0\cup\dots\cup W_{i-1}$ let $f_i(w):=f_{i-1}(w)$, and for $w\in W_i$ let $f_i(w):=v$, where $v\in V_i^*$ is the unique vertex such that $\{L_i(w), v\}\in \mathcal{M}$.

%\begin{figure}
%\label{fig:cycle}
%\begin{center}
%\input{./pics/embedding}
%\caption{A vertex $w\in W_i$ can be mapped to a vertex $v\in V^*$ if and only if $N_{V\setminus V^*}(v)\subseteq f(L_i(w))$.}
%\end{center}
%\end{figure}

As long as we find an $\mathcal{L}_i$-matching for $1\le i\le t$ we
clearly construct a valid embedding of $H$ into $G$. It remains to
show that we can find the required matchings.

We first show that for every $1\le i\le t-1$, the auxiliary graph
$B_i$ contains an $\mathcal{L}_i$-matching.

\begin{claim}
For every $1\leq i\leq t-1$, there exists an
$\mathcal{L}_i$-matching in $B_i$.
\end{claim}
\begin{proof}
We show that Hall's condition for the existence of an
$\mathcal{L}_i$-saturating matching is satisfied. First, we show that
$|\mathcal{L}_i|=|W_i|< |V_i^*|-\frac{\varepsilon n}{16}$ for $1\le i\le t-1$. We
have
$$|V_i^*|=|V_0\cup\dots \cup V_i|-|W_0\cup\dots\cup W_{i-1}|=|W_i\cup \dots \cup W_{t} |-|V_{i+1}\cap\dots\cup V_t|$$
and therefore
$$|V_i^*| - |W_i| =|W_{i+1}\cup \dots \cup W_{t} |-|V_{i+1}\cap\dots\cup V_t|\geq |W_t|-\frac{t-i}{16t}\varepsilon n>\frac{15 \varepsilon n}{16}.$$
Thus, we have that $|\mathcal{L}_i|=|W_i|\leq |V_i^*|- \frac{15 \varepsilon n}{16}<|V_i^*|- \frac{\varepsilon n}{16}$ and the claim therefore follows by Claim~\ref{claim:small_Hall} below.
\end{proof}

\begin{claim}
\label{claim:small_Hall} For all $U\subseteq \mathcal{L}_i$ that
satisfy $|U|\le |V_i^*|- \frac{ \varepsilon n}{16}$, we have
$$|N_{B_i}(U)| \geq |U|.$$
\end{claim}
\begin{proof}
Let $U=U_0\cup\ldots\cup U_d$, where
$$U_j:=\left\{L\in U \mid |L|=j \right\}.$$
If $U_0\ne \emptyset $, then $N_{B_i}(U)=V_i^*$. Therefore, we may
assume that $U_0= \emptyset$. Pick $k$ such that $|U_k|\geq |U|/d$. We show that the lemma holds for $n$ large enough by distinguish between the following three cases:

\textbf{Case 1}: $|U_k|\le (p/2)^{-k} / 2$. It follows by property $(P2)$ that
$$|N_{B_i}(U)|\geq |N_{B_i}(U_k)| \geq  (p/2)^{k}|U_k||V_i | / 2\geq \frac{\log^{2d} n }{2^{k+1} \cdot 16 d}|U|\geq |U|.$$

\textbf{Case 2}: $(p/2)^{-k} / 2 \le |U_k| \le (p/2)^{-k}
\log^{2(d-1)}n$. We fix an arbitrary subset $U_k'\subset U_k$ of
size $|U_k'| = (p/2)^{-k} / 2$, and by the same argument as in Case
1 we get that
$$|N_{B_i}(U)|\geq |N_{B_i}(U_k')| \geq (p/2)^{k}|U_k'||V_i | / 2\geq \frac{\log^{2d} n}{2^{k+1} \cdot 16}|U_k'|\geq \frac{\log n}{2^{k+1} \cdot 16 d }|U|\geq |U|.$$

\textbf{Case 3}: $ |U_k|\geq (p/2)^{-k} \log^{2(d-1)}n$. In this case note that the induced subgraph
$B_i[U_k, V_i^*\setminus N_{B_i}(U_k)]$ has no edges. By property
$(P2)$ this yields that
$$|V_i^*\setminus N_{B_i}(U_k)|<  (p/2)^{-k} \log^{2(d-1)} n =o(\varepsilon n),$$
which implies that $|N_{B_i}(U_k)|\geq |V_i^*|- o( \varepsilon n)\geq|V_i^*|- \frac{ \varepsilon n}{16} \geq |U|$.
\end{proof}

In the last lemma of this section we show that $B_t$ contains a perfect
matching, thus we can complete the embedding of $H$.
\begin{lemma}
There exists a perfect matching in $B_t$.
\end{lemma}
\begin{proof}
We check Hall's condition for every subset $U\subseteq
\mathcal{L}_t.$ For sets of cardinality $|U|\le |V_t^*|- \frac{ \varepsilon n}{16}$,
Hall's condition follows by Claim~\ref{claim:small_Hall}. Therefore,
consider only subsets $U$ of cardinality $|U|\geq |V_t^*|- \frac{ \varepsilon n}{16}$.
Let $U\subseteq \mathcal{L}_t$ be such a subset. Note that by the
definition of the partial embedding $f_0$, every set in $U$ is contained in one of the
cliques in $\mathcal{K}$. Suppose first that $|V_t^*\setminus
N_{B_t}(U)|\geq (p/2)^{-d} / 2$. We fix a subset $Y\subset
V_t^*\setminus N_{B_t}(U)$ of size exactly $(p/2)^{-d} / 2$. It
follows by property $(P1)$ that at least $2^{-d-2}\cdot
p^d|Y|\varepsilon n$ of the cliques in $\mathcal{K}$ are completely
connected to some vertices in $Y$. We conclude that
$$|N_{B_t}(V_t^*\setminus N_{B_t}(U))| \geq |N_{B_t}(Y)|\geq 2^{-3}(p/2)^d  (p/2)^{-d} \cdot \varepsilon n > \varepsilon n/ 16,$$
which is not possible since $|U|+|N_{B_t}(V_t^*\setminus
N_{B_t}(U))|\le |V_t^*|$.

Therefore, we conclude that $|V_t^*\setminus N_{B_t}(U)|\le
(p/2)^{-d} / 2$. Now, using Property $(P1)$ similarly as above we obtain
that
$$|N_{B_t}(V_t^*\setminus N_{B_t}(U))|\geq 2^{-d-3}(p/2)^d \cdot \varepsilon n|V_t^*\setminus N_{B_t}(U)| >|V_t^*\setminus N_{B_t}(U)|.$$
Finally, since
$$|N_{B_t}(U)|+|V_t^*\setminus N_{B_t}(U)|=|V_t^*|=|W_t|\geq |U|+ |N_{B_t}(V_t^*\setminus N_{B_t}(U))|> |U|+ |V_t^*\setminus N_{B_t}(U)|$$
we get $|N_{B_t}(U)|>|U|$.
\end{proof}

\section{Proof of Theorem~\ref{rainbow}}
\label{sec:uni_proof2}
In this section we prove Theorem~\ref{rainbow}. Before starting the
proof, it will be convenient to introduce the following notation. For any bipartite graph $G=(A\cup B, E)$ with $|A|=|B|=n$ and minimum degree $\delta(G)\geq k$,
let $\mathcal B^{\ell}_{k-out}(G)$ denote the following set
of bipartite graphs: each $D \in \mathcal B^{\ell}_{k-out}(G)$ has vertex set $V(D)=V(G)$ and edge set $E(D)\subseteq E$ such that each vertex in $A$ has degree exactly $k$. Note that we can sample an element from $\mathcal B^{\ell}_{k-out}(G)$ uniformly at random by choosing for each $v\in A$ uniformly at random $k$ edges from $E_G(v, B)$.

One of the main ingredients in the proof of Theorem~\ref{rainbow} is
the following simple lemma on the existence of perfect matchings
in typical graphs from $\mathcal B^{\ell}_{k-out}(G)$.

\begin{lemma}\label{lem:matching}
Let $\varepsilon>0$, let $n$ be a sufficiently large integer and let $k=\omega(\log n)$. Then for any bipartite graph $G=(A\cup B, E)$ with $|A|=|B|=n$ and $\delta(G)\geq \frac{n}{2}+\varepsilon n$, a graph $D$
chosen uniformly at random from $\mathcal B^{\ell}_{k-out}(G)$ w.h.p.\
contains a perfect matching.
\end{lemma}
\begin{proof}
Let $D$ be a graph chosen uniformly at random from $\mathcal B^{\ell}_{k-out}(G)$. We show that w.h.p.\ all subsets $S\subset A$ and all subsets $S\subset B$ with $|S|\le n/2$ satisfy $|S|\le |N_D(S)|$. It then follows from Hall's theorem (see \cite{West} for more details) that $D$ has a perfect matching.

We first assume that $S\subset A$. Note that $|S|>|N_D(S)|$ implies that there exists a subset $S'\subset B$ of size $|S'|=|S|-1$ such that $|E_D(S, B\setminus S')|=0$. Note that in $G$, since $|S'|\le n/2$, every vertex $v\in S$ has at least $\varepsilon n$ neighbors in $B\setminus S'$. Therefore, when choosing the $i$-th of the $k$ edges incident to $v$ and conditioning on the event that no edge in $E_G(v, B\setminus S')$ has been selected so far, the probability to miss $B\setminus S'$ is at most
$$\frac{\deg_G(v)-\varepsilon n -i+1}{\deg_G(v)-i+1}\le 1-\varepsilon.$$
Thus,
\begin{align*}
\Pr\left[|S|>|N_D(S)|\right]& \le  \Pr\left[\exists S'\subset B\middle| |E_D(S, B\setminus S')|=0\right]\\
&\le \binom{n}{|S|-1}(1-\varepsilon)^{|S|k}\le e^{-\varepsilon |S|\cdot  \omega(\log n)},
\end{align*}
and the probability that such a bad set exists is at most
$$\sum_{s=1}^{n/2}\binom{n}{s}   e^{-\varepsilon s\cdot  \omega(\log n)}\le \sum_{s=1}^{n/2} e^{- s\cdot  \omega(\log n)}=o(1).$$

Next, assume that $S\subset B$ and observe that in order to have $|S|>|N_D(S)|$, there must exist a set $S'\subset A$ of size $|S|-1$ such that $|E_D(A\setminus S', S)|=0$. Note that $|E_G(A,S)|\geq |S|\cdot \left(\frac{n}{2}+\varepsilon n\right)$, $|E_G(S', S)|\le |S| \cdot |S'|\le |S|\cdot \frac{n}{2}$ and therefore $|E_G(A\setminus S',S)|\geq |S|\varepsilon n $. Since every edge of $G$ appears in $D$ with probability at least $\frac{k}{n}$ (but not independently) and since this probability can only decrease if we know that another edge does not appear in $D$, it follows that
\begin{align*}
\Pr[|S|>|N_D(S)|]&\le \Pr\left[\exists S'\subset A\middle| |E_D(A\setminus S',S)|=0\right] \\
&\le \binom{n}{|S|-1}\left(1-\frac{\omega(\log n)}{n}\right)^{|S|\varepsilon n}\le  e^{-\varepsilon |S|\cdot  \omega(\log n)}
\end{align*}
as in the previous case.
\end{proof}

Now we are ready to prove Theorem \ref{rainbow}.

\begin{proof} Our proof is motivated by ideas of Cooper and Frieze \cite{CF}. Note that containing a rainbow copy of some fixed graph $H$ is a monotone increasing property and we can therefore fix $p$ to exactly $n^{-1/d}\log^{5/d}n$.
%Lemma~\ref{SimplePartitionLemma}.

Let $\Delta$ and $d$ be positive integers, let $n$ be a sufficiently
large integer and let $H\in \mathcal H(n,\Delta,d)$. Moreover, let
$\bar{d}=\frac{2|E(H)|}{n}$ denote the average degree of $H$ (note that
$\bar{d}\leq d$ and in fact can be much smaller than $d$) and let
$\alpha>0$ be some arbitrarily small positive constant. First, we show how
to partition $H$ in such a way that will later help us to find a
rainbow copy of it in a typical member of $\gnpk$, where
$c=(1+\alpha)|E(H)|$. For this aim we act as follows. If $H$
contains a set $W$ of $\lceil \frac{\alpha n}{5\log^2 n}\rceil$ isolated vertices (that is,
vertices of degree $0$ in $H$), then partition $V(H)=\{w_1\}
\cup \ldots \cup \{w_t\} \cup W$ in such a way that for each $i$, the vertex
$w_i$ has at most $d$ neighbors in $\{w_1,\ldots,w_{i-1}\}$. Indeed,
such a partition exists since $H':=H-W\in \mathcal
H(n-|W|,\Delta,d)$, and therefore is $d$-degenerate, so one can apply
Observation~\ref{Degeneracy}. Otherwise, let $x$ denote the number of vertices of degree larger than $0$ and at most
$\bar{d}$ in $H$. Since $H$ contains at most
$ \frac{\alpha n}{5\log^2 n}$ isolated vertices, the following inequality holds:

$$\bar{d}n=2|E(H)|\geq x+(\bar{d}+1)\left(n- \frac{\alpha n}{5\log^2 n}-x\right).$$

Hence, using the fact that
$n$ is sufficiently large, we conclude that $x\geq n/(2\bar{d})$.
Now, let $S$ be the set consisting of all these vertices. By applying
Lemma~\ref{IndependentSetInGraphWithBoundedDegree} to $H$ and $S$ it
follows that there exists a subset $T\subseteq S$, such that $T$ is
$2$-independent and
$$|T|\geq \frac{|S|}{\bar{d}\Delta^2}\geq
\frac{n}{2\bar{d}^2\Delta^2}\geq \left\lceil \frac{\alpha n}{5\log^2 n}\right\rceil$$
 for sufficiently large $n$. Next, let $W\subseteq T$ be an arbitrary subset of size $\lceil \frac{\alpha n}{5\log^2 n}\rceil$,
and partition $V(H)=\{w_1\} \cup \ldots \cup \{w_t\} \cup W$ in such
a way that for each $i$, $w_i$ has at most $d$ neighbors in
$\{w_1,\ldots,w_{i-1}\}$.

All in all, we have a partition $V(H)=\{w_1\} \cup \ldots \cup
\{w_t\} \cup W$ such that $|W|=\lceil \frac{\alpha n}{5\log^2 n}\rceil$ and one of the
following holds:
\begin{enumerate}[$(1)$]
\item all the vertices of $W$ are isolated in $H$, or
\item $W$ is $2$-independent and consists of non-isolated vertices of degree at most $\bar{d}$.
\end{enumerate}

Note that if (2) holds then
\begin{equation}
\label{eq:few_final_edges}
|E(W,V\setminus W)|\leq \bar{d}|W|=\frac{2|E(H)|}{n}\cdot \left\lceil \frac{\alpha n}{5\log^2 n}\right\rceil <\alpha |E(H)|/(2\lceil\log^2 n\rceil),
\end{equation}
for $n$ large enough.

Now we start to describe the procedure of finding a rainbow copy of
$H$. Let $q\geq p/2$ be such that $1-p=(1-q)^2$ and present
$G\sim\gnp$ as $G=G_1\cup G_2$, where $G_1$ and $G_2$ are two graphs
sampled independently from $\mathcal G(n,q)$. We sample a member of
$\gnpk$ by sampling a member of $\gnp$ and randomly coloring exposed
edges using $c$ colors.

We find a rainbow embedding of $H$ in $G\sim \gnpk$ in two phases.
In Phase I, we find a rainbow embedding $f$ of $H[\{w_1\cup \ldots
\cup w_{t}\}]$ with edges which are taken from $G_1$. If $W$ is as
in $(1)$ (that is, all the vertices in $W$ are isolated in $H$),
then we are done. Otherwise, in Phase II we show that one can extend
$f$ to a rainbow embedding of $H$ in $G$, using edges of $G_2$.

In what follows, we present the exact strategies of Phases I and II
and prove that w.h.p.\ everything works out well.

{\bf Phase I:} Throughout this phase we maintain a partial rainbow
embedding $f$ of $H$ to $G_1$, a set of \emph{available colors}
$\mathcal C$ and a set of \emph{available vertices} $V'$. Initially,
set $f=\emptyset$, $\mathcal C :=[c]$ and $V':=V(G)$. Additionally, we maintain for each
vertex $v\in V(G)$ a set $U_v \subseteq V(G)$ such that $U(v)\cap V'$ contains only \emph{unexposed}
potential neighbors of $v$ in $G_1$. Initially, $U_v=V(G)\setminus
\{v\}$ for each $v\in V(G)$.

We inductively build the desired partial embedding $f$ as follows.
In the first step, let $f(w_1) := v$ for an arbitrary vertex $v \in
V'$, and set $V':=V'\setminus \{v\}$. Assume that we have already
embedded $\{w_1, \ldots, w_{i-1}\}$ for some $2 \leq i \leq t$ and
we wish to embed $w:=w_i$. Let $L(w_i) = f(N_H(w_i) \cap \{w_1,
\ldots, w_{i-1}\})$ be the set of images of neighbors of $w_i$ which
have already been embedded (recall that  $|L(w_i)| \leq d$). Let
$A_w = V' \cap \left(\cap_{v \in L(w_i)} U_v\right)$ be the set of
all available vertices which are still unexposed neighbors of all
vertices in $L(w_i)$, and choose an arbitrary subset $S_w \subset
 A_w$ of size $s:= \lceil \alpha n / (4\Delta \log n)^2 \rceil$
(Claim~\ref{claim1:rainbow} shows that throughout Phase I this is
indeed possible; that is, $A_w$ is of size at least $s$). Expose all
edges between $L(w_i)$ and $S_w$, and assign uniformly at random
colors to all the obtained edges. Let $x \in S_w$ be a vertex which
is connected to all the vertices in $L(w_i)$ and such that all the
colors assigned to edges $\{ vx \mid v\in L(w_i)\}$ are distinct and belong to $\mathcal C$. The existence of such a vertex
follows from Claim~\ref{claim2:rainbow} below. We extend $f$ by
defining $f(w_i):=x$, update $U_v:=U_v\setminus S_w$ for all $v \in L(w_i)$, $V':=V'\setminus \{x\}$ and
$$\mathcal{C}:=\mathcal C\setminus\{col \in \mathcal{C} \mid \exists
v \in L(w_i) \text{ such that } vx \text{ is colored in }
col\}.$$

The following two claims show that w.h.p.\ we manage to find the
desired embedding in Phase~I.

\begin{claim}\label{claim1:rainbow}
Throughout Phase I we have that $|A_w| \geq  \lceil \alpha n / (4\Delta \log n)^2 \rceil$ for every vertex $w\in V(H)$ which has not been embedded.
\end{claim}

\begin{proof} The proof of the claim is obtained from the following four observations. First, note that at the beginning of Phase I we have that
$U_v=V(G)\setminus \{v\}$ for each $v\in V(G)$. Second, we update
$U_v$ only after embedding a vertex $w$ for which $v \in
L(w)$ (and then we delete the set $S_w$ which is of size $s =
 \lceil \alpha n / (4\Delta \log n)^2 \rceil$ from $U_v$). Third, every vertex $v$ is
a member of at most $\Delta$ sets $L(w)$ (recall that $\Delta(H)
\leq \Delta$). Fourth, note that $|V'|\geq \lceil \alpha n /(5\log^2 n)\rceil $ throughout
Phase I (recall that we do not embed $W$ in this phase).

Therefore, it follows that at any point during Phase I we have
$$|U_v\cap V'|\geq |V'|-1-\Delta\cdot\left\lceil\frac{\alpha n}{ (4 \Delta\log n)^2}\right\rceil,$$
for each vertex $v \in V(G)$. Since $|L(w)|\leq \Delta$, we conclude that
$$|A_w|=|V'\cap \left(\cap_{v\in
L(w)}U_v\right)| \geq |V'| - \Delta - \Delta^2\left\lceil\frac{\alpha n}{(4 \Delta \log n )^2}\right\rceil \geq  s,$$
for $n$ large enough.
\end{proof}

The next claim states that whenever we wish to embed a vertex $w$,
it has at least one candidate in $V'$.

\begin{claim}\label{claim2:rainbow}

Let $w\in V(H)\setminus W$. At the
moment we try to embed $w$ there exists with probability $1-o(1/n)$ a vertex $x\in S_w$ for which the following holds:
\begin{enumerate}[$(i)$]
\item $x$ is connected to all the vertices in $L(w)$,
and
\item all the colors assigned to the edges $\{\{v,x\}:v\in L(w)
\}$ are distinct and belong to $\mathcal
C$.
\end{enumerate}
\end{claim}

\begin{proof}
Let
$$X:=|\{v\in S_w \mid L(w)\subseteq N_{G_1}(v)|\}|.$$
Note that $X$ is the sum of i.i.d. indicator random variables $X_v$
(for all $v\in S_w$) for which $X_v=1$ iff $L(w)\subseteq
N_{G_1}(v)$. Clearly, we have that (recall that $|L(w)|\leq d$)
$$\mathbb{E}[X]\geq s q^{d} \geq \frac{\alpha n}{(4 \Delta \log n)^2} \cdot \Omega\left(\frac{\log^5
n}{n}\right)=\Omega(\log^3 n).$$ Applying Chernoff's bound we obtain
that
$$\Pr[X \le \mathbb{E}[X]/2]=e^{- \Omega(\log^3 n)} = o(1/n).$$
Now, note that $|\mathcal{C}|\geq \alpha |E(H)|$ during Phase I.
Thus, the probability that for a vertex $x\in S_w$ with $L(w)\subseteq
N_{G_1}(x)$, all the edges to $L(w)$ have different colors from
$\mathcal{C}$ is at least
$$\frac{\binom{\mathcal{C}}{\ell}}{\left((1 + \alpha)|E(H)|\right)^{\ell}}\geq \left(\frac{\alpha |E(H)|}{(1 + \alpha)|E(H)| \ell}\right)^{\ell} \geq \left( \frac{\alpha}{(1 + \alpha)d} \right)^d=:\gamma>0,$$
where $|L(w)|=\ell$. Therefore, if $X\geq \mathbb{E}[X]/2$ then the probability that
there is no such $x$ is at most
$$ \left(1 - \gamma \right)^{|X|} \leq e^{-\gamma |X|}= e^{- \Omega(\log^3 n)} = o(1/n).$$
\end{proof}

Note that since we embed at most $n$ vertices, applying the union
bound we obtain that for every vertex $w_i$ there exists a ``good"
vertex $x \in S_w$. Now, if $W$ is as in $(1)$ (that is, all the
vertices in $W$ are isolated in $H$), then we are done. Otherwise, we
continue to Phase II.

%\begin{claim}\label{claim2:rainbow}
%For every $w\in V(H)$, at the moment we try to embed in Phase I,
%w.h.p.\ there are $\omega(\log n)$ vertices in $S_w$ which are
%connected to all vertices in $L(w)$. Moreover, there are at least
%$2$ such vertices $x$ and $y$ for which all the colors assigned to
%the edges $\{vs:v\in L(w) \text{ and } s\in \{x,y\}\}$ are distinct
%and belong to $\mathcal C$.
%\end{claim}
%
%\begin{proof} Note that the random variable $X_w$ which counts the number of vertices in $S_w$ which are
%connected to all the vertices in $L(w)$ is binomial distributed with
%$X_w\sim Bin(s,p^t)$, where $t=|L(w)|\leq d$. Therefore, using
%Chernoff inequality we obtain that $\Pr(X_w\leq
%sp^t/2)=e^{-\Theta(sp^t)}=o(1/n)$. Applying the union bound we
%obtain that in every $S_w$, there are at least $sp^t/2=\omega(\log
%n)$ vertices which are connected to all the vertices in $L(w)$. Now,
%assigning colors to the obtained edges, using the fact that at any
%point during Phase I we used at most $dn$ colors from $\mathcal C$,
%it follows that the probability that all the colors of edges between
%a fixed $s\in S_w$ and $L(w)$ are distinct and belong to $\mathcal
%C$ is at least $$\frac{\binom{|\mathcal C|}{t}}{(Cn)^t}\geq
%\frac{\binom{Cn-dn}{t}}{(Cn)^t}\geq
%\left(\frac{C-d}{Ct}\right)^t\geq \gamma.$$
%
%Therefore, the probability that there will not be $x$ and $y$ as
%required is $\gamma^{\omega(\log n)}=o(1/n)$. Applying the union
%bound we get the desired.
%\end{proof}

{\bf Phase II:} Let $V^*:=V(G)\setminus f(V(H)\setminus W)$. Our goal is to extend $f$ with a valid embedding of $W$ into $V^*$,
using edges of $G_2$, in such a way that the resulting embedding is
rainbow.

For $w\in W$ let $L(w):=f(N_H(w))$ and let $\mathcal{L} = \left\{L(w)\middle|w\in W\right\}$. Recall that $W$ is
$2$-independent and thus all the $L(w)$'s are disjoint. %Next, we define an auxiliary bipartite graph $\mathcal{B}(\mathcal L,S)$ as explained in the second paragraph of Section~\ref{sec:universality}.
Let $F=(\mathcal{L}\cup V^*, E_F)$ with edge set
$$E_F:= \left\{ Lv  \mid L\in\mathcal{L},  v\in V^* \text{ and } \forall_{u\in L} uv \notin E(G_1) \right\}$$
be the ground graph to build a bipartite auxiliary graph $\mathcal{B}(\mathcal{L}, V^*)$. Edges that appeared in $G_1$ are excluded since we can not color them again. Note that $|\mathcal{L}|=|W|=|V^*|$ and that by the following very rough estimate $F$ satisfies w.h.p.\ the conditions of Lemma~\ref{lem:matching}.

 \begin{claim}
 It holds with high probability that $\delta(F)\geq \frac{3}{4}|V^*|$.
 \end{claim}
 \begin{proof}
%First assume that there exists an $L\in \mathcal{L}$ with $deg_F(L)<\frac{3}{4}|W|$. This means that there is a set $S\subseteq W$ of cardinality $|S|>\frac{1}{4}|W|$ such that in $G_1$, every vertex in $S$ has a neighbor in $L$. By counting the edges between $S$ and $L$, we conclude that there exist a vertex $v\in L$ such that $deg_{G_1}(v)>\frac{1}{4d}|W|\geq \frac{\delta n}{20d \log^2 n}$. It is easy to see that w.h.p.\ every vertex in $G_1$ has degree at most $2nq<<\frac{n}{20d\log^2 n}$.
%The symmetric argument shows that w.h.p.\ every $v\in W$ has degree at least $\frac{3}{4}|W|$ in $F$.
%\end{proof}
%\begin{proof}
For every $L\in\mathcal{L}$ and $v\in V^*$ the edge $Lv\notin E_F$ if and only if there exists $u\in L$ for which $uv\in E(G_1)$. Since $G_1\sim G(n,q)$, by applying Chernoff's bound it follows that w.h.p.\ $\Delta(G_1)\le 2nq$. Moreover, since for every $L\in\mathcal{L}$ we have that $|L|\le \bar{d}$, it follows that $deg_F(L)\geq |V^*|-\bar{d}2nq>3|V^*|/4$. A similar argument shows that we have $deg_F(v)\geq 3|V^*|/4$ for every $v\in V^*$.
 \end{proof}

In the following we describe a random process that tries to create a bipartite graph $\mathcal{B}(\mathcal{L},W)\in \mathcal{B}_{\lceil \log^2n\rceil- out}^{\ell}(F)$ by exposing edges from $G_2\setminus G_1$ and randomly color them.
First, let
$$\mathcal{C}:=\left\{col \in [c] \mid \exists\{u,v\}\in E(H\setminus W) \text{ s.t. } \{f(u), f(v)\} \text{ has color $col$}\right\} $$
and note that $|\mathcal{C}|\geq \alpha |E(H)|$.
Choose an arbitrary ordering $L_1, \dots, L_{|\mathcal{L}|}$ of the elements in $\mathcal{L}$. Then, in step $1\le i\le |\mathcal{L}|$, set $N_i:=N_F(L_i)$ and create $\lceil\log^2 n\rceil$ edges from $L_i$ to vertices in $N_i$ as follows: as long as $|N_{\mathcal{B}(\mathcal{L}, V^*)}(L_i)|<\lceil \log^2n\rceil$, iteratively pick a vertex $v\in N_i$ uniformly at random, set $N_i:=N_i\setminus\{v\}$ and expose all edges from $v$ to vertices in $L_i$ and color them uniformly at random with colors from $[c]$ (note that here the process can fail if at some point $N_i=\emptyset$ while $|N_{\mathcal{B}(\mathcal{L}, V^*)}(L_i)|<\lceil \log^2n\rceil$). If all the edges are contained in $G_2$ and if they have distinct colors that are all from the set of available colors $\mathcal{C}$, add $L_iv$ to $\mathcal{B}(\mathcal{L}, V^*)$.
At the end of step $i$ remove all the colors used at edges incident to $L_i$,
$$\mathcal{C}:=\mathcal{C}\setminus \left\{col\in [c]\mid\exists u\in L_i, \exists v\in N_{\mathcal{B}(\mathcal{L}, V^*)}(L_i) \text{ s.t. } uv \text{ has color } col\right\}.$$

If the process succeeds then every matching $M$ in $\mathcal{B}(\mathcal{L}, V^*)$ is clearly rainbow in the sense that all edges in
$$\left\{uv\mid \exists Lv\in M \text{ s.t. } u\in L \right\}$$
have distinct colors that have not been used in the embedding in Phase I. It follows from Claim~\ref{claim3:rainbow} below and Lemma~\ref{lem:matching} that the process succeeds w.h.p.\ and that the constructed $\mathcal B(\mathcal L,V^*)$ contains a perfect matching. Finally,  such a perfect matching in $\mathcal B(\mathcal L,V^*)$
extends $f$ into a rainbow embedding of $H$ in $G$. The following claim therefore completes the proof of Theorem~\ref{rainbow}.

\begin{claim}\label{claim3:rainbow}
The random process that creates $\mathcal{B}(\mathcal{L}, V^*)$ in Phase II succeeds w.h.p.\ and it samples uniformly at random from $\mathcal{B}_{\lceil\log^2n\rceil-out}^{\ell}(F)$.
\end{claim}
\begin{proof}
Note first that the process can only fail if in some round $1\le i\le |\mathcal{L}|$ we have that $N_i=\emptyset$ and $|N_{\mathcal{B}(\mathcal{L}, V^*)}(L_i)|<\lceil \log^2n\rceil$. It therefore suffices to show that in a fixed step $1\le i\le |\mathcal{L}|$ the process creates with probability $1-o(1/n)$ the $\lceil \log^2 n\rceil$ required edges.
Let
$$X_i:=\{v\in N_{F}(L_i) \mid L_i\subseteq N_{G_2}(v)|\}.$$
Note that $|X_i|$ is the sum of i.i.d. indicator random variables $X_{i,v}$
(for all $v\in N_{F}(L_i)$) for which $X_{i,v}=1$ iff $L_i\subseteq
N_{G_2}(v)$. Clearly, we have that (recall that $|L_i|\leq \bar{d} \le d$)
$$\mathbb{E}[|X_i|]\geq |N_{F}(L_i)| q^{d} \geq \delta(F) \cdot \Omega\left(\frac{\log^5
n}{n}\right)\geq \frac{3}{4}\cdot \frac{\alpha n}{5\log^2 n}\cdot \Omega\left(\frac{\log^5
n}{n}\right)=\Omega(\log^3 n).$$ Applying Chernoff's bound we obtain
that
\begin{equation}
\label{eq:pr_x_i}\Pr[|X_i| \le \mathbb{E}[|X_i|]/2]=e^{- \Omega(\log^3 n)} = o(1/n).
\end{equation}
Next, let
$$Y_i:=\{ v\in X_i \mid \text{all edges in $E(L_i, v)$ have distinct colors from $\mathcal{C}$} \}.$$
Note that $|Y_i|$ is the sum of i.i.d. indicator variables $Y_{i,v}$ (for all $v\in X_i$) for which $Y_{i,v}=1$ iff all edges in $E(L_i, v)$ have distinct colors from $\mathcal{C}$. Since we have by \eqref{eq:few_final_edges} that $|E(W, V\setminus W)|\le \alpha |E(H)|/(2\lceil\log^2 n\rceil)$ and we remove for each edge in $E(W, V\setminus W)$ at most $\lceil \log^2 n\rceil$ colors from $\mathcal{C}$,  the number of available colors in $\mathcal{C}$ is always at least $\alpha|E(H)|/2$.
Thus, the probability that for a vertex $v \in X_i$ all the edges to $L_i$ have different colors from
$\mathcal{C}$ is at least
$$p_{i}=\frac{\binom{\mathcal{C}}{\ell}}{\left((1 + \alpha)|E(H)|\right)^{\ell}}\geq \left(\frac{\alpha |E(H)| / 2}{(1 + \alpha)|E(H)| \ell}\right)^{\ell} \geq \left( \frac{\alpha}{(1 + \alpha)2\bar{d}} \right)^{\bar{d}}=:\gamma>0,$$
where $|L_i|=\ell\le \bar{d}$, and this lower bound for $p_{i}$ holds independently of all other color assignments in previous steps.
Therefore, if $|X_i|\geq \mathbb{E}[|X_i|]/2$, then the expectation of $|Y_i|$ is at least
$$\mathbb{E}[|Y_i|]\geq |X_i|\cdot \gamma=\Omega(\mathbb{E}[|X_i|])=\Omega(\log^3 n)$$
and it follows from Chernoff's bound that
\begin{equation}
\label{eq:p_y_i}
\Pr\left[|Y_i|<\frac{\mathbb{E}[|Y_i|]}{2}\middle| |X_i|\geq \frac{\mathbb{E}[|X_i|]}{2}\right]  =e^{- \Omega(\log^3 n)} = o(1/n).
\end{equation}
Combining \eqref{eq:pr_x_i} and \eqref{eq:p_y_i} we conclude that the probability that our process fails is at most
$$\sum_{i=1}^{|\mathcal{L}|}\Pr\left[|Y_i|\le \lceil \log^2 n\rceil\right]\le |\mathcal{L}| \cdot o(1/n) =o(1).$$
Finally, since we choose a random ordering of the neighbors of $L_i$, every $\lceil \log^2n\rceil$-tuple of neighbors of $L_i$ has the same probability to be part of $\mathcal{B}(\mathcal{L}, V^*)$ and the process therefore samples an element of $\mathcal{B}_{\lceil\log^2n\rceil- out}^{\ell}(F)$ uniformly at random.
\end{proof}

%As in
%Phase I, it  follows from Claim~\ref{claim2:rainbow} and union bound
%that we find w.h.p.\ for each $L\in \mathcal{L}_i$ two neighbors in
%$S_L$ and for each $v\in V_i$ two neighbours in $S_v$, such that all
%of the edges between these elements are of distinct colors from
%$\mathcal{C}'$. From symmetry it is clear that $B(\mathcal L_i,V_i)$
%is distributed as $\mathcal D_{2-out}(|\mathcal L_i|,|V_i|,c)$ and
%from Theorem \ref{Walkup} it follows that w.h.p.\ it contains a
%perfect matching $\mathcal M$. Therefore, one can extend the
%embedding $f$ to $W_i$ by defining $f(w)=v$ for every edge
%$\{L(w),v\}\in \mathcal M$. This completes the proof.
\end{proof}

{\bf Acknowledgment.} The first author is grateful to Michael
Krivelevich for pointing out the problem of finding a rainbow
embedding in random graphs, and to Benny Sudakov for many helpful
and valuable conversations. The authors are also grateful to Peter
Allen for giving some useful comments on an earlier draft. Last but not least, the authors are grateful to the anonymous referees for many valuable comments.

\bibliographystyle{amsplain}
\bibliography{refs}

\end{document}